\documentclass{amsart}
\usepackage{amsthm, amssymb, amsfonts, amscd}
\usepackage{graphicx}
\usepackage{tikz-cd}
\usepackage{verbatim}
\usepackage{enumitem}
\usepackage{mathrsfs, mathtools}
\usepackage{xcolor}
\usepackage[colorlinks = true,
            linkcolor = black,
            urlcolor  = blue,
            citecolor = black]{hyperref}
\usepackage{cleveref}
\usepackage[margin=1in]{geometry}
\usepackage{setspace}
\allowdisplaybreaks

\theoremstyle{definition} 
\newtheorem{thm}{Theorem}[section]
\newtheorem{cor}[thm]{Corollary}

\newtheorem{lem}[thm]{Lemma}
\newtheorem{conj}[thm]{Conjecture}

\newtheorem{exmp}[thm]{Example}

\newcommand{\N}{\mathbb{N}}
\newcommand{\Z}{\mathbb{Z}}
\newcommand{\Q}{\mathbb{Q}}

\newcommand{\C}{\mathbb{C}}
\newcommand{\A}{\mathbb{A}}
\newcommand{\F}{\mathbb{F}}
\newcommand{\G}{\mathbb{G}}
\begin{document}

\title[A converse of DML in positive characteristic]{A converse of dynamical Mordell--Lang conjecture in positive characteristic}
\author{Jungin Lee and GyeongHyeon Nam}
\date{}
\address{J. Lee -- Department of Mathematics, Ajou University, Suwon 16499, Republic of Korea \newline
\indent G. Nam -- Department of Mathematics, Ajou University, Suwon 16499, Republic of Korea}
\email{jileemath@ajou.ac.kr, gyeonghyeon.alg@gmail.com}

\begin{abstract}
In this paper, we prove the converse of the dynamical Mordell--Lang conjecture in positive characteristic: For every subset $S \subseteq \mathbb{N}_0$ which is a union of finitely many arithmetic progressions along with finitely many $p$-sets of the form $\left \{ \sum_{j=1}^{m} c_j p^{k_jn_j} : n_j \in \N_0 \right \}$ ($c_j \in \mathbb{Q}$, $k_j \in \mathbb{N}_0$), there exist a split torus $X = \mathbb{G}_m^k$ defined over $K=\overline{\mathbb{F}_p}(t)$, an endomorphism $\Phi$ of $X$, $\alpha \in X(K)$ and a closed subvariety $V \subseteq X$ such that $\left \{ n \in \N_0 : \Phi^n(\alpha) \in V(K) \right \} = S$.
\end{abstract}
\maketitle

\section{Introduction} \label{Sec1}

Throughout the paper, let $p$ be a prime, $\N$ be the set of positive integers and $\N_0 := \N \cup \left \{ 0 \right \}$. An \textit{arithmetic progression} in $\N_0$ is a set of the form $\left \{ a+bt : t \in \N_0 \right \}$ for some $a, b \in \N_0$. We use a convention that a \textit{variety} may not be irreducible. 

The \textit{dynamical Mordell--Lang conjecture} predicts that for a quasi-projective variety $X$ defined over a field $K$ of characteristic $0$, an endomorphism $\Phi$ of $X$, a point $\alpha \in X(K)$ and a closed subvariety $V \subseteq X$, the set of $n \in \N_0$ such that $\Phi^n(\alpha) \in V(K)$ is a finite union of arithmetic progressions in $\N_0$. It is one of the central problems in arithmetic dynamics. 

\begin{conj}\label{conj:dml}
(Dynamical Mordell--Lang conjecture; \cite[Conjecture 1.7]{GT09}) Let $X$ be a quasi-projective variety defined over a field $K$ of characteristic $0$, $\Phi$ be an endomorphism of $X$, $\alpha \in X(K)$ and $V \subseteq X$ be a closed subvariety. Then the set $\left \{ n \in \N_0 : \Phi^n(\alpha) \in V(K) \right \}$ is a union of finitely many arithmetic progressions.
\end{conj}

We remark that if Conjecture \ref{conj:dml} holds for $K = \C$, then it holds for every field $K$ of characteristic $0$ \cite[Proposition 3.1.2.1]{BGT16}. There is extensive literature on various special cases of dynamical Mordell--Lang conjecture. We provide two significant examples.
\begin{thm}
    Conjecture \ref{conj:dml} is true for the following cases:
    \begin{enumerate}
    \item (Bell--Ghioca--Tucker \cite[Theorem 1.3]{BGT10}) $\Phi:X\rightarrow X$ is an \'etale morphism over $\C$,
    \item (Xie \cite[Theorem 1]{Xie17}, \cite[Theorem 1.4]{Xie23}) $\Phi:\A^2\rightarrow \A^2$ is a morphism over $\C$.
\end{enumerate}
\end{thm}

It is natural to ask whether Conjecture \ref{conj:dml} still holds when $K$ has a positive characteristic. The following example says that the assumption that $K$ has a characteristic $0$ in Conjecture \ref{conj:dml} is essential. 

\begin{exmp} \label{exmp1a}
(\cite[Example 1.2]{Ghi19}) Let $p$ be an odd prime, $X = \G_m^3$ be defined over $K=\F_p(t)$, $\Phi : X \rightarrow X$ be given by $\Phi(x,y,z)=(tx, (1+t)y, (1-t)z)$, $\alpha = (1,1,1) \in X(K)$ and $V \subset \G_m^3$ be the hyperplane defined by the equation $y+z-2x=2$. Then $\left \{ n \in \N_0 : \Phi^n(\alpha) \in V(K) \right \} = \left \{ p^{n_1}+p^{n_2} : n_1, n_2 \in \N_0 \right \}$.
\end{exmp}

Motivated by the work of Moosa and Scanlon \cite[Theorem B]{MS04}, Ghioca and Scanlon proposed a version of dynamical Mordell--Lang conjecture in positive characteristic. 

\begin{conj} \label{conj1b}
(Ghioca--Scanlon; \cite[Conjecture 13.2.0.1]{BGT16}) Let $K$ be a field of characteristic $p$, $X$ be a quasi-projective variety defined over $K$, $\Phi$ be an endomorphism of $X$, $\alpha \in X(K)$ and $V \subseteq X$ be a closed subvariety. Then the set $\left \{ n \in \N_0 : \Phi^n(\alpha) \in V(K) \right \}$ is a union of finitely many arithmetic progressions along with finitely many $p$\textit{-sets} of the form
\begin{equation} \label{eq1a}
\left \{ \sum_{j=1}^{m} c_j p^{k_jn_j} : n_j \in \N_0 \text{ for each } 1 \le j \le m \right \} \subset \N_0
\end{equation}
for some $m \in \N$, $c_1, \ldots, c_m \in \Q$ and $k_1, \ldots, k_m \in \N_0$.
\end{conj}

Conjecture \ref{conj1b} has been studied by several authors \cite{BGT16, Ghi19, CGSZ21, Yan23}. Notably, it has been proved in \cite{CGSZ21} when $X = \G_m^N$ and $V \subset \G_m^N$ has dimension at most $2$ \cite[Theorem 1.2]{CGSZ21}, or when $X = \G_m^N$ and $\Phi : \G_m^N \rightarrow \G_m^N$ is an algebraic group endomorphism such that no iterate of $\Phi$ induces a power of Frobenius on a nontrivial connected algebraic subgroup of $\G_m^N$ \cite[Theorem 1.3]{CGSZ21}. 

For a fixed field $K$ of characteristic $p$, we say $S \subseteq \N_0$ is a $K$-\textit{DML set} if there are $X$, $V$, $\Phi$ and $\alpha$ defined over $K$ as in Conjecture \ref{conj1b} such that $S = \left \{ n \in \N_0 : \Phi^n(\alpha) \in V(K) \right \}$. For a quasi-projective variety $X$ defined over $K$, We say $S \subseteq \N_0$ is a $K$-\textit{DML set over} $X$ if there are $V$, $\Phi$ and $\alpha$ defined over $K$ such that $S = \left \{ n \in \N_0 : \Phi^n(\alpha) \in V(K) \right \}$. We say $S \subseteq \N_0$ is a $K$-\textit{DML set over split torus} if it is a $K$-DML set over $\G_m^k$ for some $k \in \N$. 

For $m \in \N$, $c_1, \ldots, c_m \in \Q \setminus \left \{ 0 \right \}$, $k_1, \ldots, k_m \in \N_0$ and a prime power $q = p^k$, we denote
$$
B(q; c_1, \ldots, c_m; k_1, \ldots, k_m) := \left \{ \sum_{j=1}^{m} c_j q^{k_jn_j} : n_j \in \N_0  \text{ for each } 1 \le j \le m \right \}.
$$
We say a subset $S \subseteq \N_0$ is a $p$-\textit{set} if $S = B(p; c_1, \ldots, c_m; k_1, \ldots, k_m)$ for some $m \in \N$, $c_1, \ldots, c_m \in \Q \setminus \left \{ 0 \right \}$ and $k_1, \ldots, k_m \in \N_0$. Note that since we have
$$
B(q; c_1, \ldots, c_m; k_1, \ldots, k_m) = B(p; c_1, \ldots, c_m; kk_1, \ldots, kk_m)
$$
for $q=p^k$, the set $B(q; c_1, \ldots, c_m; k_1, \ldots, k_m)$ is a $p$-set if it is a subset of $\N_0$. Corvaja, Ghioca, Scanlon and Zannier \cite{CGSZ21} proved that a $p$-set $B(q; c_1, \ldots, c_m; k_1, \ldots, k_m)$ is a $\F_q(t)$-DML set over split torus when $c_j$ are small positive integers and $k_j=1$. (Take $u_n=n$ in \cite[Theorem 1.4]{CGSZ21}). Note that \cite[Theorem 1.4]{CGSZ21} is a direct consequence of \cite[Proposition 4.2 and 4.3]{CGSZ21} so we can replace the field $K$ with $\F_q(t)$.

\begin{thm} \label{thm1c}
(\cite[Theorem 1.4]{CGSZ21}) A $p$-set $B(q; c_1, \ldots, c_m; 1, \ldots, 1)$ is a $\F_q(t)$-DML set over split torus if $c_j$ are positive integers such that $\sum_{j=1}^{m} c_j < \frac{q}{2}$. 
\end{thm}

The goal of this paper is to prove the converse of the dynamical Mordell--Lang conjecture in positive characteristic, i.e., every union of finitely many arithmetic progressions in $\N_0$ along with finitely many $p$-sets is a $K$-DML set for some field $K$ of characteristic $p$. Now we state the main theorem of the paper. 

\begin{thm} \label{mainthm}
Let $S$ be a subset of $\N_0$ which is a union of finitely many arithmetic progressions along with finitely many $p$-sets. Then $S$ is a $\overline{\F_p}(t)$-DML set over split torus.
\end{thm}

\subsection{Organization}
The paper is organized as follows. In Section \ref{Sec2}, we reduce the proof of Theorem \ref{mainthm} to a special form (Theorem \ref{mainthm_red}) using several lemmas. 
Section \ref{Sec3} is devoted to the proof of the main theorem. We prove Theorem \ref{mainthm_red} by dividing each $p$-set into smaller $p$-sets to which Theorem \ref{thm1c} can be applied and using lemmas in Section \ref{Sec2}.

\bigskip
\section{Reduction of the main theorem} \label{Sec2}
In this section, we provide a reduction of Theorem \ref{mainthm} to a simpler problem. From now on, we assume that $K$ is a field of characteristic $p$ and $K \neq \F_2$.

\begin{lem} \label{lem2a}
If $S_i$ is a $K$-DML set over $X_i$ for $1 \le i \le m$, then $\cup_{i=1}^{m} S_i$ and $\cap_{i=1}^{m} S_i$ are $K$-DML sets over $X_1 \times \cdots \times X_m$.
\end{lem}

\begin{proof}
It is enough to prove the lemma for $m=2$. Assume that $S_i = \left \{ n \in \N_0 : \Phi_i^n(\alpha_i) \in V_i(K) \right \}$ ($V_i \subseteq X_i$) is a $K$-DML set over $X_i$ for $i \in \left \{ 1, 2 \right \}$. Then we have 
$$
S_1 \cup S_2 = \left \{ n \in \N_0 : (\Phi_1, \Phi_2)^n(\alpha_1, \alpha_2) \in V(K) \right \}
$$ 
for $V = V_1 \times X_2 \cup X_1 \times V_2 \subseteq X_1 \times X_2$ (a variety is not assumed to be irreducible) and
$$
S_1 \cap S_2 = \left \{ n \in \N_0 : (\Phi_1, \Phi_2)^n(\alpha_1, \alpha_2) \in V(K) \right \}
$$ 
for $V = V_1 \times V_2 \subseteq X_1 \times X_2$. Thus $S_1 \cup S_2$ and $S_1 \cap S_2$ are $K$-DML sets over $X_1 \times X_2$. 
\end{proof}

\begin{lem} \label{lem2b}
Every arithmetic progression in $\N_0$ is a $K$-DML set over split torus.
\end{lem}

\begin{proof}
We prove that an arithmetic progression $T_{a,b} := \left \{ a+bt : t \ge 0 \right \}$ ($a, b \in \N_0$) is a $K$-DML set over split torus. Since $K \ne \F_2$, there exists an element $c \in K \setminus \left \{ 0, 1 \right \}$. 

\begin{enumerate}
    \item $b=0$ : Let $X = \G_m^{a+1}$, $\alpha = (1, c, \ldots, c)$, $V = \left \{ (x_1, \ldots, x_{a+1}) \in X : x_{a+1}=1 \right \}$ and $\Phi : X \rightarrow X$ be given by $\Phi(x_1, \ldots, x_{a+1}) = (c, x_1, \ldots, x_{a})$. Then we have $\left \{ n \in \N_0 : \Phi^n(\alpha) \in V(K) \right \} = \left \{ a \right \} = T_{a,0}$ so $T_{a,0}$ is a $K$-DML set over $X$.
    
    \item $b>0$ : Since $T_{a,b} = T_{a,0} \cup T_{a+b, b}$, we may assume that $a>0$ by Lemma \ref{lem2a}. Let $X = \G_m^{a+b}$, $\alpha = (c, 1, \ldots, 1)$, $V = \left \{ (x_1, \ldots, x_{a+b}) \in X : x_1 = \cdots = x_{a+b-1}=1 \right \}$ and $\Phi : X \rightarrow X$ be given by 
    $$\Phi(x_1, \ldots, x_a, y_1, \ldots, y_b) = (1, x_1, \ldots, x_{a-1}, y_2, \ldots, y_b, x_ay_1).$$ 
    Then we have $\left \{ n \in \N_0 : \Phi^n(\alpha) \in V(K) \right \} = T_{a,b}$ so $T_{a,b}$ is a $K$-DML set over $X$. \qedhere
\end{enumerate}
\end{proof}

By Lemma \ref{lem2a} and \ref{lem2b}, it remains to show that every $p$-set $B(p; c_1, \ldots, c_m; k_1, \ldots, k_m)$ is a $\overline{\F_p}(t)$-DML set over split torus. The following lemma will be frequently used in the paper. For $S \subseteq \N_0$ and $m \in \N$, write $mS := \left \{ ms : s \in S \right \}$ and $m+S := \left \{ m+s : s \in S \right \}$.

\begin{lem} \label{lem2c}
Let $S \subseteq \N_0$ and $m \in \N$. 
\begin{enumerate}
    \item $mS$ is a $K$-DML set over split torus if and only if $S$ is a $K$-DML set over split torus.
    \item $m+S$ is a $K$-DML set over split torus if and only if $S$ is a $K$-DML set over split torus.
\end{enumerate}
\end{lem}

\begin{proof}
\begin{enumerate}
    \item Assume that $mS$ is a $K$-DML set over $X=\G_m^k$. If $\left \{ n \in \N_0 : \Phi^n(\alpha) \in V(K) \right \} = mS$, then 
$$
\left \{ n \in \N_0 : (\Phi^m)^n(\alpha) \in V(K) \right \} = \left \{ n \in \N_0 : mn \in mS \right \} = S
$$
so $S$ is also a $K$-DML set over $X=\G_m^k$. Conversely, assume that $S$ is a $K$-DML set over $X = \G_m^k$ and $\left \{ n \in \N_0 : \Phi^n(\alpha) \in V(K) \right \} = S$. Define $X' = X^m = \G_m^{km}$, $V'=V \times X^{m-1}$, $\Phi' : X' \rightarrow X'$ ($\Phi'(x_1, \ldots, x_m) = (\Phi(x_m), x_1, \ldots, x_{m-1})$) and $\alpha' = (\alpha, \ldots, \alpha) \in X'(K)$. Then
$$
\left \{ n \in \N_0 : (\Phi')^n(\alpha') \in V'(K) \right \} \cap T_{0,m} = mS
$$
so $mS$ is a $K$-DML set over split torus by Lemma \ref{lem2a} and \ref{lem2b}.

\item Assume that $m+S$ is a $K$-DML set over $X=\G_m^k$. If $\left \{ n \in \N_0 : \Phi^n(\alpha) \in V(K) \right \} = m+S$, then
$$
\left \{ n \in \N_0 : \Phi^n(\Phi^{m}(\alpha)) \in V(K) \right \} = S
$$
so $S$ is also a $K$-DML set over $X=\G_m^k$. 
Conversely, assume that $S$ is a $K$-DML set over $X=\G_m^k$ and $\left \{ n \in \N_0 : \Phi^n(\alpha) \in V(K) \right \} = S$. 
If $S = \N_0$, then $m+S = T_{m, 1}$ is a $K$-DML set over split torus. Now assume that $S \neq \N_0$ and choose $\beta \in X(K) \setminus V(K)$. 
Define $X' = X^2 = \G_m^{2k}$, $V' = X \times V$, $\Phi' : X' \rightarrow X'$ ($\Phi'(x, y) = (\Phi(x), x)$) and $\alpha' = (\alpha, \beta) \in X'(K)$. Then 
$$
\left \{ n \in \N_0 : (\Phi')^n(\alpha') \in V'(K) \right \} = 1+S,
$$
so $1+S$ is a $K$-DML set over $X'$. Iterating this process, we conclude that $m+S$ is a $K$-DML set over split torus. \qedhere
\end{enumerate}
\end{proof}

By Lemma \ref{lem2c}(1), we may assume that $c_1, \ldots, c_m \in \Z \setminus \left \{ 0 \right \}$. By Lemma \ref{lem2c}(2), we may assume that $k_1, \ldots, k_m > 0$. In addition, since the sum $\sum_{j=1}^{m} c_j p^{k_jn_j}$ is nonnegative for every $n_1, \ldots, n_m \in \N_0$, we have $c_j \in \N$ for every $j$. Now for $k := k_1 \cdots k_m$, we have
\begin{equation*}
B(p; c_1, \ldots, c_m; k_1, \ldots, k_m) = \bigcup_{\substack{0 \le i_t \le k/k_t-1 \\ \text{for } 1 \le t \le m}} B(p; c_1p^{k_1i_1}, \ldots, c_mp^{k_mi_m} ; k, \ldots, k). 
\end{equation*}

\bigskip
\section{Proof of Theorem \ref{mainthm}} \label{Sec3}

In this section, we prove Theorem \ref{mainthm}. By the reduction arguments in Section \ref{Sec2}, it is enough to prove the following special form of Theorem \ref{mainthm}.

\begin{thm} \label{mainthm_red}
Every $p$-set 
$$
B(p; c_1, \ldots, c_m; k, \ldots, k) = B(q; c_1, \ldots, c_m; 1, \ldots, 1)
$$
($c_j \in \N$, $k \in \N$, $q=p^k$) is a $\overline{\F_p}(t)$-DML set over split torus.
\end{thm}

The key idea of the proof is to divide the $p$-set $B(q; c_1, \ldots, c_m; 1, \ldots, 1)$ into smaller $p$-sets to which Theorem \ref{thm1c} can be applied. 

\begin{lem} \label{lem3a}
Let $S$ be a $K$-DML set over $X$ and $L$ be an extension field of $K$. Then $S$ is a $L$-DML set over $X_L := X \times_K L$.
\end{lem}

\begin{proof}
Assume that $\left \{ n \in \N_0 : \Phi^n(\alpha) \in V(K) \right \} = S$ for $X$, $V$, $\Phi$ and $\alpha$. Then 
$X_L := X \times_K L$, $V_L := V \times_K L$, $\Phi_L : X_L \rightarrow X_L$ and $\alpha \in X(K) \subseteq X(L) = X_L(L)$ satisfies
$$
\left \{ n \in \N_0 : \Phi_L^n(\alpha) \in V_L(L) \right \} 
= \left \{ n \in \N_0 : \Phi^n(\alpha) \in V(L) \cap X(K) \right \} 
= S
$$
since $V(L) \cap X(K) = V(K)$ (cf. \cite[Remark 3.31]{Liu06}). 
\end{proof}

\begin{cor} \label{cor3b}
A $p$-set $B(q; c_1, \ldots, c_m; 1, \ldots, 1)$ is a $\overline{\F_p}(t)$-DML set over split torus if $c_j \in \N$ and $\sum_{j=1}^{m} c_j < \frac{q}{2}$.
\end{cor}

\begin{proof}
$B(q; c_1, \ldots, c_m; 1, \ldots, 1)$ is a $\F_q(t)$-DML set over split torus by Theorem \ref{thm1c}. Since $\overline{\F_p}(t)$ is an extension field of $\F_q(t)$, $B(q; c_1, \ldots, c_m; 1, \ldots, 1)$ is a $\overline{\F_p}(t)$-DML set over split torus by Lemma \ref{lem3a}.
\end{proof}

The following lemma is an important ingredient in the proof of Theorem \ref{mainthm_red}.

\begin{lem} \label{lem3c}
Assume that Theorem \ref{mainthm_red} holds for $m = m_0$. If $B(q; c_1, \ldots, c_{m_0+1}; 1, \ldots, 1)$ is a $\overline{\F_p}(t)$-DML set over split torus, then $B(q; c_1, \ldots, c_{m_0}, qc_{m_0+1}; 1, \ldots, 1)$ is also a $\overline{\F_p}(t)$-DML set over split torus.
\end{lem}

\begin{proof}
We have
\begin{align*}
& B(q; c_1, \ldots, c_{m_0}, qc_{m_0+1}; 1, \ldots, 1) \\
= & \bigcup_{i=1}^{m_0} \left ( c_i + B(q; c_1, \ldots, c_{i-1}, c_{i+1}, \ldots, c_{m_0}, qc_{m_0+1}; 1, \ldots, 1) \right ) \cup B(q; qc_1, \ldots, qc_{m_0+1}; 1, \ldots, 1) \\
= & \bigcup_{i=1}^{m_0} \left ( c_i + B(q; c_1, \ldots, c_{i-1}, c_{i+1}, \ldots, c_{m_0}, qc_{m_0+1}; 1, \ldots, 1) \right ) \cup qB(q; c_1, \ldots, c_{m_0+1}; 1, \ldots, 1).
\end{align*}
By Lemma \ref{lem2c}(2) and the assumption that Theorem \ref{mainthm_red} holds for $m=m_0$, each set 
$$
c_i + B(q; c_1, \ldots, c_{i-1}, c_{i+1}, \ldots, c_{m_0}, qc_{m_0+1}; 1, \ldots, 1)
$$ 
is a $\overline{\F_p}(t)$-DML set over split torus. By Lemma \ref{lem2c}(1), $qB(q; c_1, \ldots, c_{m_0+1}; 1, \ldots, 1)$ is a $\overline{\F_p}(t)$-DML set over split torus. Thus $B(q; c_1, \ldots, c_{m_0}, qc_{m_0+1}; 1, \ldots, 1)$ is a $\overline{\F_p}(t)$-DML set over split torus by Lemma \ref{lem2a}(1).
\end{proof}

\begin{proof}[Proof of Theorem \ref{mainthm_red}]
When $m=1$, we have 
$$
B(q; c_1; 1)=c_1B(q; 1; 1)=c_1B(q^2; 1; 1) \cup qc_1B(q^2; 1; 1).
$$
The $p$-set $B(q^2; 1; 1)$ is a $\overline{\F_p}(t)$-DML set over split torus by Corollary \ref{cor3b} so $B(q; c_1; 1)$ is a $\overline{\F_p}(t)$-DML set over split torus by Lemma \ref{lem2a} and \ref{lem2c}(1).

We proceed by induction on $m$ and assume that Theorem \ref{mainthm_red} holds for $m-1 \in \N$. Choose a sufficiently large $k \in \N$ such that $c_1 + \cdots + c_m < \frac{q^{\frac{k}{m}}}{2}$. Then we have
$$
B(q; c_1, \ldots, c_m; 1, \ldots, 1)=\underset{\substack{0 \leq i_t \leq k-1\\ \text{for } 1 \le t \le m}}{\bigcup} B(q^k; c_1q^{i_1}, \ldots, c_mq^{i_m}; 1, \ldots, 1)
$$
so it is enough to show each each $p$-set $B(q^k; c_1q^{i_1}, \ldots, c_mq^{i_m}; 1, \ldots, 1)$ is a $\overline{\F_p}(t)$-DML set over split torus. For $i_{\min} := \min(i_1, \ldots, i_m)$, we have 
$$
B(q^k; c_1q^{i_1}, \ldots, c_mq^{i_m}; 1, \ldots, 1) = q^{i_{\min}} B(q^k; c_1q^{i_1-i_{\min}}, \ldots, c_mq^{i_m-i_{\min}}; 1, \ldots, 1)
$$
so we may assume that $\min(i_1, \ldots, i_m)=0$ by Lemma \ref{lem2c}(1). Without loss of generality, assume that $k = i_0 \ge i_1 \ge \cdots \ge i_m =0$. 

Since $\sum_{t=0}^{m-1} (i_t-i_{t+1}) = i_0-i_m = k$, we have $\max(i_0-i_1, \ldots, i_{m-1}-i_m) \ge \frac{k}{m}$. Choose $0 \le r \le m-1$ such that $i_r-i_{r+1} = \max(i_0-i_1, \ldots, i_{m-1}-i_m) \ge \frac{k}{m}$. Then we have
    $$
    q^{k-i_r}B(q^k; c_1q^{i_1}, \ldots, c_mq^{i_m}; 1, \ldots, 1) = B(q^k; c_1q^{k-i_r+i_1}, \ldots, c_mq^{k-i_r+i_m}; 1, \ldots, 1).
    $$
    By Lemma \ref{lem2c}(1), the induction hypothesis and Lemma \ref{lem3c}, the $p$-set
    $$
    B(q^k; c_1q^{i_1}, \ldots, c_mq^{i_m}; 1, \ldots, 1)
    $$
    is a $\overline{\F_p}(t)$-DML set over split torus if the $p$-set
    $$
    B(q^k; c_1q^{i_1-i_r}, \ldots, c_{r-1}q^{i_{r-1}-i_r}, c_r, c_{r+1}q^{k-i_r+i_{r+1}}, \ldots, c_mq^{k-i_r+i_m}; 1, \ldots, 1)
    $$
    is a $\overline{\F_p}(t)$-DML set over split torus.
    Now we have 
\begin{align*}
& c_1q^{i_1-i_r} + \cdots + c_{r-1}q^{i_{r-1}-i_r} + c_r + c_{r+1}q^{k-i_r+i_{r+1}} + \cdots + c_mq^{k-i_r+i_m} \\
\le \, & q^{k-i_r+i_{r+1}}(c_1 + \cdots + c_m) \\
< \, & q^{k-\frac{k}{m}} \cdot \frac{q^{\frac{k}{m}}}{2} \\
= \, & \frac{q^k}{2}    
\end{align*}
so the $p$-set $ B(q^k; c_1q^{i_1}, \ldots, c_mq^{i_m}; 1, \ldots, 1)$ is a $\overline{\F_p}(t)$-DML set over split torus by Corollary \ref{cor3b}. 
\end{proof}

\bigskip
\section*{Acknowledgments}

J. Lee was supported by the new faculty research fund of Ajou University (S-2023-G0001-00236). The authors thank Dragos Ghioca for helpful comments.



\begin{thebibliography}{99}
\bibitem{BGT10}
J. P. Bell, D. Ghioca and T. J. Tucker, The dynamical Mordell--Lang problem for \'etale maps, Amer. J. Math. 132 (2010), 1655--1675.

\bibitem{BGT16}
J. P. Bell, D. Ghioca and T. J. Tucker, The dynamical Mordell--Lang conjecture, Mathematical Surveys and Monographs 210, American Mathematical Society, Providence, RI, 2016.

\bibitem{CGSZ21}
P. Corvaja, D. Ghioca, T. Scanlon, and U. Zannier, The dynamical Mordell--Lang conjecture for endomorphisms of semiabelian varieties defined over fields of positive characteristic, J. Inst. Math. Jussieu 20 (2021), no. 2, 669--698.

\bibitem{Ghi19}
D. Ghioca, The dynamical Mordell--Lang conjecture in positive characteristic, Trans. Amer. Math. Soc. 371 (2019), no. 2, 1151--1167.

\bibitem{GT09}
D. Ghioca and T. J. Tucker. Periodic points, linearizing maps, and the dynamical Mordell--Lang problem. J. Number Theory 129 (2009), no. 6, 1392--1403.

\bibitem{Liu06}
Q. Liu, Algebraic geometry and arithmetic curves, Oxford Graduate Texts in Mathematics, vol. 6, Oxford University Press, Oxford, 2006.

\bibitem{MS04}
R. Moosa and T. Scanlon, $F$-structures and integral points on semiabelian varieties over finite fields, Amer. J. Math. 126 (2004), 473--522.

\bibitem{Xie17}
J. Xie, The dynamical Mordell--Lang conjecture for polynomial endomorphisms of the affine plane, Ast\'erisque 394 (2017), vi+110.

\bibitem{Xie23}
J. Xie, Around the dynamical Mordell--Lang conjecture, arXiv:2307.05885.

\bibitem{Yan23}
S. Yang, Dynamical Mordell--Lang conjecture for totally inseparable liftings of Frobenius, arXiv:2302.00948.
\end{thebibliography}
\end{document}